\documentclass[oneside]{amsart}
    \usepackage{amsmath}
    \usepackage{amsfonts}
    \usepackage{amsthm}
    \usepackage{amssymb}
    \usepackage{hyperref}
    \usepackage{amscd}
    \usepackage{color}
    \usepackage{enumerate}
    \usepackage{xypic}
    \usepackage{tikz}
    \usepackage{tikz-cd}
    \usepackage{colortbl}
    \usepackage[normalem]{ulem}
    \usepackage{multirow}
    \usepackage{makecell}

    \usepackage{booktabs}

    \usepackage{dsfont}
    
    \makeatletter
    \def\l@subsection{\@tocline{2}{0pt}{2.5pc}{5pc}{}}
    \makeatother

    \newcommand{\val}{\mathfrak{v}}
    \newcommand{\cS}{\mathcal{S}}

    \newcommand{\Aff}{\mathbb{A}}
    
    \newcommand{\ZZ}{\mathbb{Z}}
    \newcommand{\KK}{\mathbb{K}}
    \newcommand{\RR}{\mathbb{R}}
    \newcommand{\PP}{\mathbb{P}}

    \newcommand{\NN}{\mathbb{N}}

    \newcommand{\Gm}{\mathbb{G}_m}

    \DeclareMathOperator{\PGL}{PGL}
    \DeclareMathOperator{\GL}{GL}
    \DeclareMathOperator{\In}{In}

    \newtheorem{thm}[equation]{Theorem}
    
    \newtheorem{lemma}[equation]{Lemma}

    \theoremstyle{definition}
    
    \newtheorem{defn}[equation]{Definition}

    \newcommand{\hide}[1]{}
    
    \numberwithin{equation}{section}

    \title{Toric Degenerations of Low Degree Hypersurfaces}
    
    \author{Nathan Ilten}
    \address{Department of Mathematics, Simon Fraser University,
    8888 University Drive, Burnaby BC V5A1S6, Canada}
    \email{\href{mailto:nilten@sfu.ca}{nilten@sfu.ca}}
    
    \author{Oscar Lautsch}
    \address{Department of Mathematics, Simon Fraser University,
    8888 University Drive, Burnaby BC V5A1S6, Canada}
    \email{\href{mailto:oscar\_lautsch@sfu.ca}{oscar\_lautsch@sfu.ca}}

\begin{document}
   \maketitle
   \begin{abstract}
We show that a sufficiently general hypersurface of degree $d$ in $\PP^n$ admits a toric Gr\"obner degeneration after linear change of coordinates if and only if $d\leq 2n-1$.
   \end{abstract}

    \section{Introduction}
When does a projective variety $X$ admit a flat degeneration to a toric variety? 
Among other applications, such degenerations are used in the mirror-theoretic approach to the classification of Fano varieties \cite{fanosearch}, the construction of integrable systems \cite{integrable}, and in bounding Seshadri constants \cite{seshadri}.
The many applications of toric degenerations notwithstanding, there is as of yet no general method for determining if a given variety admits a toric degeneration.

In this note, we will consider the special case of toric degenerations of some $X\subset \PP^n$ obtained as the flat limit of $X$ under a $\Gm$-action on $\PP^n$. In the case that the $\Gm$-action arises as a one-parameter subgroup of the standard torus on $\PP^n$, the situation may be well understood by studying the Gr\"obner fan and tropicalization of $X$ \cite{tropical}. 
However, if we consider arbitrary $\Gm$-actions on $\PP^n$, the situation becomes more complicated. As a test case, we investigate the existence of such toric degenerations when $X$ is a hypersurface.

In order to state our result, we introduce some notation.
Throughout the paper, $\KK$ will be an algebraically closed field of characteristic zero.
Let $\omega\in\RR^{n+1}$. Consider any polynomial $f\in\KK[x_0,\ldots,x_n]$, 
where we write 
\begin{equation}\label{eqn:f}
	f=\sum_{u\in \ZZ_{\geq 0}^{n+1}}  c_ux^u
\end{equation}
using multi-index notation.
The \emph{initial term} of $f$ with respect to the weight vector $\omega$ is 
\[
\In_\omega(f)=\sum_{u: \langle u,\omega \rangle =\lambda} c_ux^u
\]
where $\lambda$ is the maximum of $\langle u,\omega \rangle$ as $u$ ranges over all $u\in \ZZ_{\geq 0}$ with $c_u\neq 0$.
For an ideal $J\subset \KK[x_0,\ldots,x_n]$, its initial ideal with respect to the weight vector $\omega$ is 
\[
\In_\omega(J)=\langle \In_\omega(f)\ |\ f\in J\rangle.
\]
The \emph{weight} of a monomial $x^u$ with respect to $\omega$ is the scalar product $\langle u,\omega \rangle\in \RR$.
\begin{defn}
	Let $X\subset \PP^n$ be a projective variety over $\KK$. We say that \emph{$X$ admits a toric Gr\"obner degeneration up to change of coordinates} if there exists a $\PGL(n+1)$ translate $X'$ of $X$ and a weight vector $\omega\in\RR^{n+1}$ such that the initial ideal
	\[
\In_\omega (I(X'))
	\]
	of the ideal $I(X')\subseteq \KK[x_0,\ldots,x_n]$ of $X'$ is a prime binomial ideal.
\end{defn}

We can now state our result:
\begin{thm}\label{thm:main}
	Let $d,n\in \NN$.
There is a non-empty Zariski open subset $U$ of the linear system of degree $d$ hypersurfaces in $\PP^n$ with the property that every hypersurface in $U$ admits a toric Gr\"obner degeneration up to change of coordinates if and only if $d\leq 2n-1$.
\end{thm}

\noindent Before proving this theorem in the following section, we discuss connections to the existing literature.

	A common source of toric degenerations of a projective variety $X\subset \PP^n$ arises by considering the Rees algebra associated to a full-rank homogeneous valuation $\val$ on the homogeneous coordinate ring of $X$ \cite{anderson}. As long as the homogeneous coordinate ring of $X$ contains a finite set $\cS$ whose valuations generate the value semi-group of $\val$, one obtains a toric degeneration. Such a set $\cS$ is called a \emph{finite Khovanskii basis} for the coordinate ring of $X$. This construction is in fact quite general: essentially any $\Gm$-equivariant degeneration of $X$ over $\Aff^1$ arises by this construction, see \cite[Theorem 1.11]{kmm} for a precise statement.
There has been some work  on algorithmically constructing valuations with finite Khovanskii bases (see e.g.~\cite{flag} for applications to degenerations of certain flag varieties) but as of yet, there is no general effective criterion for deciding when such a valuation exists.

	Drawing on \cite{km} which connects Khovanskii bases and tropical geometry, we may rephrase our results in the language of Khovanskii bases. It is straightforward to show that $X$ admits a toric Gr\"obner degeneration up to change of coordinates if and only if there is some full-rank homogeneous valuation $\val$ for which the homogeneous coordinate ring has a finite Khovanskii basis consisting of degree one elements. Thus, our theorem shows the existence of finite Khovanskii bases for general hypersurfaces of degree at most $2n-1$, and shows that any finite Khovanskii basis for a general hypersurface of larger degree necessarily contains elements of degrees larger than one. In fact, we suspect that a general hypersurface of sufficiently large degree does not admit any finite Khovanskii basis at all.

We note in passing that a general hypersurface of arbitrary degree will admit a toric degeneration in a weaker sense. Indeed, the universal hypersurface over the linear system of degree $d$ hypersurfaces is a flat family, and for any degree $d$ there is a toric hypersurface of degree $d$. However, such a degeneration is not $\Gm$-equivariant.

An interesting comparison of our result can be made with \cite{kaveh2021generic}, which states that after a \emph{generic} change of coordinates, any arithmetically Cohen Macaulay variety $X\subset \PP^n$ has a Gr\"obner degeneration to a (potentially non-normal) variety equipped with an effective action of a codimension-one torus. Such varieties, called complexity-one $T$-varieties, are in a sense one step away from being toric. The hypersurfaces we consider in our main result (Theorem \ref{thm:main}) are of course arithematically Cohen Macaulay, so they admit Gr\"obner degenerations to complexity-one $T$-varieties.
Our result characterizes when we can go one step further and Gr\"obner degenerate to something toric. When $d\leq 2n-1$ and we are in the range for which this is possible for a generic hypersurface, the change of coordinates required is a special one as opposed to the generic change of coordinates of \cite{kaveh2021generic}.

	\subsection*{Acknowledgements}
	N.~Ilten was supported by an NSERC Discovery Grant. O.~Lautsch was supported by an NSERC undergraduate student research award.

    \section{Proof of Theorem}
    \subsection*{Setup}
    Throughout, we will assume that $d,n>1$ since the theorem is clearly true if $d=1$ or $n=1$.
    We will view the coefficients $c_u$ of $f$ in \eqref{eqn:f} as coordinates on affine space $\Aff^{d+n \choose n}$. To indicate the dependence of $f$ on the choice of coefficients $c$, we will often write $f=f_c$.
    Let $K$ be the subset of all $u\in\ZZ_{\geq 0}^{n+1}$ such that $u_0+u_1=d$, $u_i=0$ for $i>1$, and $u_1<d$. We then set
    \[
	    W=V(\langle c_u \rangle_{u\in K})\subset \Aff^{d+n \choose n}.
    \]
The family of polynomials parameterized by $W$ consists of all degree $d$ forms such that the only monomial involving only $x_0$ and $x_1$ is $x_1^d$.

We will be considering the map
\begin{align*}
	\phi:\GL(n+1)\times W&\to \KK[x_0,\ldots,x_n]_d\\
		(A,c)&\mapsto A.f_c
\end{align*}
where $A.f_c$ denotes the action of $A\in \GL(n+1)$ on a polynomial $f_c=\sum c_ux^u$ via linear change of coordinates. We will be especially interested in the differential of $\phi$ at $(e,c)$, where $e\in\GL(n+1)$ is the identity. A straightforward computation shows that the image of the differential at $(e,c)$ is generated by 
\begin{align}
	&x^u\qquad &u\notin K;\label{eqn:d1}\\
	&\frac{\partial f_c}{\partial x_i} \cdot x_j \qquad &0\leq i,j \leq n.\label{eqn:d2}
\end{align}

The following lemma is the key to our proof:
\begin{lemma}\label{lemma:key}
The differential $\phi$ is surjective at $(e,c)$ for general $c\in W$ if and only if $d\leq 2n-1$.
\end{lemma}
\begin{proof}
Consider the image of the differential of $\phi$ at $(e,c)$.
From \eqref{eqn:d1} we obtain the span of all monomials of $\KK[x_0,\ldots,x_n]_d$ with the exceptions of the $d$ monomials $x_0^d,x_0^{d-1}x_1,\ldots,x_0x_1^{d-1}$. 
From \eqref{eqn:d2} with $i=1$ and $j=0$, modulo \eqref{eqn:d1} we additionally obtain the monomial $x_0x_1^{d-1}$.
We do not obtain anything new from \eqref{eqn:d2} when $i=1$ and $j=1$, when $i=0$, or when $j>1$.

It remains to consider the contributions to the image from \eqref{eqn:d2} with $i>1$ and $j=0,1$.
For $2\leq i \leq n$ and $1\leq m\leq d-1$, let $u(i,m)\in \ZZ^{n+1}$ be the exponent vector with $u_i=1$, $u_0=m$, $u_1=d-m-1$.
Modulo the span of \eqref{eqn:d1} and $x_0x_1^{d-1}$, from \eqref{eqn:d2} we obtain
\begin{align*}
	\frac{\partial f_c}{\partial x_i}\cdot x_0\equiv	&c_{u(i,d-1)}x_0^d+c_{u(i,d-2)}x_0^{d-1}x_1+\ldots+c_{u(i,1)}x_0^2x_1^{d-2}\\
	\frac{\partial f_c}{\partial x_i}\cdot x_1\equiv	&\phantom{c_{u(i,d-1)}x_0^d+}c_{u(i,d-1)}x_0^{d-1}x_1+\ldots+c_{u(i,2)}x_0^2x_1^{d-2}.
\end{align*}
Varying $i$ from $2$ to $n$, we obtain $2n-2$ polynomials of degree $d$. 
The $(2n-2)\times (d-1)$ matrix of their coefficients has the form
\[
	\left(\begin{array}{c c c  c c}
c_{u(2,d-1)}&c_{u(2,d-2)}&\ldots&c_{u(2,1)}\\
0&c_{u(2,d-1)}&\ldots&c_{u(2,2)}\\
c_{u(3,d-1)}&c_{u(3,d-2)}&\ldots&c_{u(3,1)}\\
0&c_{u(3,d-1)}&\ldots&c_{u(3,2)}\\
\vdots&\vdots&&\vdots\\
c_{u(n,d-1)}&c_{u(n,d-2)}&\ldots&c_{u(n,1)}\\
0&c_{u(n,d-1)}&\ldots&c_{u(n,2)}
	\end{array}
		\right)
\]

Since $c\in W$ is general, this matrix has full rank, that is, 
its rank is $\min\{d-1,2n-2\}$. 
Hence, the image of the differential of $\phi$ has codimension
\[
d-1-\min\{d-1,2n-2\},
\]
so the differential is surjective if and only if $d\leq 2n-1$.
\end{proof}

We now move on to prove the theorem.

\subsection*{Existence}
We will first show that if $d\leq 2n-1$, a general degree $d$ hypersurface admits a toric Gr\"obner degeneration up to change of coordinates. 
 As noted above,
the family of polynomials parameterized by $W$ consists of all degree $d$ forms such that the only monomial involving only $x_0$ and $x_1$ is $x_1^d$. Consider any $\omega\in\RR^{n+1}$ such that 
\begin{equation*}
\omega_0>\omega_1>\omega_2>\cdots>\omega_n\qquad (d-1)\omega_0+\omega_2=d\omega_1.
\end{equation*}
For general $c\in W$, the initial term of $f_{c}$ is
\[
	ax_1^d+bx_0^{d-1}x_2
\]
for some $a,b\neq 0$;
this is a prime binomial.
Thus, we will be done with our first claim if we can show that the image of $\phi$ contains a non-empty open subset of $\KK[x_0,\ldots,x_n]_d$.

To this end, we consider the image of the differential at $(e,c)$ for general $c\in W$.
By Lemma \ref{lemma:key}, we conclude that $\phi$ has surjective differential at $(e,c)$ for general $c\in W$; it follows that $\phi$ has surjective differential at a general point of $\GL(n+1)\times W$. Thus, the dimension of the image of $\phi$ is the dimension of 
$\KK[x_0,\ldots,x_n]_d$, and the image of $\phi$ contains a non-empty open subset of $\KK[x_0,\ldots,x_n]_d$.
\subsection*{Non-existence}
Assume now that $d>2n-1$.
We first give an overview of the proof strategy. There are only finitely many prime binomials $g$ of degree $d$. Likewise, there are only finitely many linear orderings $\prec$ of the variable indices $0,\ldots,n$.
We say that a weight vector $\omega$ is \emph{compatible} with $\prec$ and $g$ if whenever $i\prec j$ in the linear ordering, then $\omega_i\geq \omega_j$, and the two monomials of $g$ have the same weight with respect to $\omega$.

For fixed $g$ and linear ordering on the variables, we may consider the set $S$ of all polynomials $f$ in $\KK[x_0,\ldots,x_n]_d$ for which there exists a compatible weight vector $\omega\in\RR^{n+1}$ such that initial term of $f$ with respect to $\omega$ is $g$.
We will show that up to permutation of the coordinates, this set $S$ can be identified as a subfamily of $W$. By Lemma \ref{lemma:key}, the map $\phi$ has nowhere surjective differential. Thus, by generic smoothness, the dimension of the image of $\phi$ must be strictly less than the dimension of $\KK[x_0,\ldots,x_n]_d$. It follows that there cannot be a Zariski-open subset of $\KK[x_0,\ldots,x_n]_d$ such that every hypersurface in this subset admits a toric Gr\"obner degeneration up to change of coordinates.

To complete the proof, we will fix a prime binomial $g=g'+g''$ of degree $d$ and a linear ordering of the variables.
Here, $g'$ and $g''$ are the two terms of $g$.
After permuting the variables and appropriately adapting $g$, we may assume without loss of generality that the indices are ordered as $0\prec 1 \prec  2 \prec \ldots  \prec n$. 
The irreducibility of $g$ implies that $g$ involves at least three distinct variables, and no variable appears in both $g'$ and $g''$.
Let $p$ be the smallest index such that $x_p$ appears in $g$; we denote the corresponding term by $g'$. Let $q$ be the smallest index such that $x_q$ appears in the term $g''$.

If $g'$ only involves variables $x_i$ with indices $i<q$, then any compatible term order $\omega$ must satisfy $\omega_p=\omega_q=\omega_j$ for all $p\leq j \leq q$. Indeed, if not, the term $g''$ would necessarily have smaller weight  Without loss of generality, we may thus permute indices without changing the set of compatible weight vectors to also assume that $g'$ involves some $x_i$ with $i>q$. For this, we are using that the irreducibility of $g$ guarantees that at least one of $g'$ and $g''$ is not a $d$th power.

Consider the set $S$ of polynomials $f_c$ such that there is a compatible weight $\omega$ for which $f_c$ has $g$ as its initial term. We claim that $S$ is a subset of the family parameterized by $W$
Indeed, since $q>0$, $g''$ has weight at most equal to the weight of $x_1^d$. The monomials $x_0^d,x_0^{d-1}x_1,\ldots,x_0x_1^{d-1}$ all have weight at least as big as the weight of $x_1^d$, and are not scalar multiples of $g'$ or $g''$. Hence, none of these monomials can appear in any element of $S$, and the claim follows.

The proof of the theorem now follows from the argument given above.
\bibliographystyle{alpha}
    \bibliography{paper}
    \end{document}